\title{\vspace{-1cm}Decay of solutions of diffusive Oldroyd-B system in $\mathbb{R}^2$}
\date{\today}
\documentclass[12pt]{article}
\usepackage{authblk}
\usepackage{blindtext}
\usepackage{amsfonts}
\usepackage{amsmath}
\usepackage{amsthm}
\usepackage{graphicx}
\usepackage[margin=2cm]{geometry}
\newtheorem{theorem}{Theorem}
\newtheorem{lemma}{Lemma}

\newtheorem{remark}{Remark}

\newtheorem{definition}{Definition}
\newcommand{\norm}[1]{\left\lVert#1\right\rVert}
\author[1]{	Joonhyun La \thanks{joonhyun@math.princeton.edu}}
\affil[1] {Department of mathematics, Princeton University}
\begin{document}
\maketitle

\abstract{We show that strong solutions of 2D diffusive Oldroyd-B systems in $\mathbb{R}^2$ decay at an algebraic rate, for a large class of initial data. The main ingredient for the proof is the following fact; an Oldroyd-B system is a macroscopic closure of a Fokker-Planck-Navier-Stokes system, and the free energy of this Fokker-Planck-Navier-Stokes system decays over time. In particular, $\norm{u}_{L^\infty _t L^2 _x}$ and $\norm{\nabla_x u}_{L^2 _t L^2 _x }$ are uniformly bounded for all time.}

\section{Introduction}

We are interested in the long-time dynamics of the following system, which is called a diffusive Oldroyd-B system. It reads the following:
\begin{equation}
\left \{
\begin{gathered}
\partial_t u (x,t) +  u(x,t) \cdot \nabla_x u (x,t) = \nu_1 \Delta_x u(x, t) - \nabla_x p (x,t) + \mu \nabla_x \cdot \tau (x,t) , \\
\nabla \cdot u = 0, \\
\partial_t \tau + u \cdot \nabla_x \tau = (\nabla_x u) \tau + \tau (\nabla_x u)^T - 2k  \tau + \rho \left ( \nabla_x u + (\nabla_x u)^T \right ) + \nu_2 \Delta_x  \tau, \\
\partial_t \rho + u \cdot \nabla \rho = \nu_2 \Delta  \rho,\\
u(x,0) = u_0 (x), \rho(x,0) = \rho_0, \tau(x,0) = \tau_0,
\end{gathered}
\right. \label{Oldroyd}
\end{equation}
where $\nu_1, k, \mu >0$, and $\nu_2 >0$ are constants representing viscosity, inter-particle diffusivity, strengh of the stress due to polymer, and center-of-mass diffusion effects, respectively. This system is widely used to model viscoelastic flows. We will consider the case of the spatial domain $\mathbb{R}^2$, although we will also briefly discuss the case of the spatial domain $\mathbb{T}^2$. In the case of $\mathbb{T}^2$, it is customary to put $\rho(x, t) \equiv 1$. \newline
The system (\ref{Oldroyd}) is extensively studied in the literature. Barrett and Boyaval proved global existence of weak solution in \cite{MR2843021}. In \cite{MR2989441}, Constantin and Kliegl proved global well-posedness of strong solution. Plan, Gupta, Vincenzi, and Gibbon calculated Lyapunov dimension of the system (\ref{Oldroyd}) under the assumption of the existence of global attractor in \cite{MR3666384}. This system is a formal macroscopic closure of Fokker-Planck-Navier-Stokes system, and a rigorous justification is provided recently by Barrett and S\"{u}li in \cite{MR3698145}. The author, in \cite{La2018}, extended the result in \cite{MR3698145} to a larger class of data and provided a rigorous proof of the fact that the free energy of the system decreases over time. \newline 
The models with $\nu_2 = 0$ are called non-diffusive models, and they are also widely studied. Guillop\'{e} and Saut proved local existence, uniqueness of strong solution, and global existence of strong solution for small initial data, in the case of bounded domain, in \cite{MR1077577} and in \cite{MR1055305}. Fern\'{a}ndez-Cara, Guill\'{e}n, and Ortega extended the results of Guillop\'{e} and Saut to $L^p$ setting in \cite{MR1422802}, \cite{MR1633055}, and \cite{MR1893419}. In addition, Hieber, Naito, and Shibata studied the system in the case of exterior domain in \cite{MR2860633}. Chemin and Masmoudi studied the system in critical Besov spaces, and proved local well-posedness of the system and provided a Beale-Kato-Majda type (\cite{MR763762}) criterion in \cite{MR1857990}. Other Beale-Kato-Majda type sufficient conditions were given by Kupferman, Mangoubi, and Titi in \cite{MR2398006}, and by Lei, Masmoudi, and Zhou in \cite{MR2558169}. In addition, Lions and Masmoudi showed global existence of weak solution for corotational models in \cite{MR1763488}. Hu and Lin proved in \cite{MR3434615} global existence of weak solution for non-corotational models, given that the initial deformation gradient is close to the identity and the initial velocity is small. In \cite{MR2165379}, Lin, Liu, and Zhang developed an approach based on deformation tensor and Lagrangian particle dynamics. Lei and Zhou studied the system via incompressible limit in \cite{MR2191777} and proved global existence for small data. Also, Lei, Liu, and Zhou studied global existence for small data and incompressible limit in \cite{MR2393434}. Moreover, in \cite{MR3473592}, Fang and Zi proved global well-posedness for initial data whose vertical velocity field can be large. Constantin and Sun proved global existence for small data with large gradients for Oldroyd-B, and considered regularization of Oldroyd-B model in \cite{MR2901300}. \newline
The study of long-time behavior of the system will rely on the Fourier splitting technique. This technique is widely used to study long-time behavior of parabolic systems (\cite{MR571048}, \cite{MR775190}, \cite{MR856876}, \cite{MR837929}, \cite{MR1356749}, \cite{MR881519}).

\section{Oldroyd-B system, Fokker-Planck-Navier-Stokes system, and Free energy estimate}

The system (\ref{Oldroyd}) can be obtained from the Fokker-Planck-Navier-Stokes system as a macroscopic closure of it: from the system
\begin{equation}
\left \{
\begin{gathered}
\partial_t u (x,t) +  u(x,t) \cdot \nabla_x u (x,t) = - \nabla_x p (x,t) + \nu_1 \Delta_x u (x, t) + \mu \nabla_x \cdot \sigma (x,t) , \\
\nabla \cdot u = 0, \\
\partial_t f(x,m,t) + u \cdot \nabla_x f + (\nabla_x u)  m \cdot \nabla_m f  = k \Delta_m f + k \nabla_m \cdot (m f ) + \nu_2 \Delta_x f, \\
\rho (x, t) =\int_{\mathbb{R}^2} f(x,m,t) dm, \\
\sigma (x,t) = \int_{\mathbb{R}^2} m \otimes m f(x,m,t) dm, \\
u(x,0) = u_0 (x), f(x,m,0 ) = f_0 (x,m),
\end{gathered}
\right. \label{Kinetic}
\end{equation}
by letting $\tau = \sigma - \rho \mathbb{I}$,  we obtain (\ref{Oldroyd}). In \cite{La2018}, we proved the following: 
\begin{theorem}[\cite{La2018}]
Suppose that
\begin{equation}
\begin{gathered}
u_0 \in \mathbb{P}W^{2,2 }, \\
f_0 \ge 0, \int_{\mathbb{R}^2_m \times \mathbb{R}^2_x} f_0 (x, m) dm dx = 1, \\
M_{a,b} [f_0] \in W^{1,2} \,\, a+b = 2p \le 6, M_{a,b} [f_0] \in L^2 \,\, a+b = 2q \le 16,  \bar{M}_4[f_0] \in L^1, \\
\int_{\mathbb{R}^2_m \times \mathbb{R}^2_x } f_0 \log f_0 dmdx < \infty, \int_{\mathbb{R}^2_x} |\Lambda (x) |^2 M_{0,0} [f_0 ] dx < \infty, \Lambda(x) = \log (\max(|x|, 1 ) ), \\
\int M_{0,0} [f_0] \log \left ( M_{0,0} [f_0] \right ) dx < \infty
\end{gathered} \label{initcondmicromacro}
\end{equation} 
where 
\begin{equation}
M_{a,b} [f_0] (x)  = \int m_1 ^a m_2 ^b f_0(m, x) dm, \,\, \bar{M}_{p} [f_0] (x) = \int |m|^p |f_0(m,x)| dm.
\end{equation}
Then there is a unique global solution for the system (\ref{Kinetic}), smooth enough for $u$ and $\sigma$, and it satisfies the following free energy estimate:
\begin{equation}
\begin{gathered}
\mu \int f(t) \log \frac{f(t)}{M_{0,0} [f(t) ] f_E } dmdx + \norm{u(t)}_{L^2} ^2 + \nu_1 \int_0 ^t \norm{\nabla_x u(s) } _{L^2} ^2 ds \\
\le \norm{u_0}_{L^2} ^2 + \mu \int f_0 \log \frac{f_0}{M_{0,0} [f_0] f_E} dmdx,
\end{gathered} \label{freeenergy}
\end{equation}
where $f_E = \frac{e^{-\frac{|m|^2}{2}}}{Z}$ is the equilibrium distribution. 
\end{theorem}
Note that given $\rho_0, \sigma_0$ satisfying 
\begin{equation}
\begin{gathered}
\rho_0 \in L^1 \cap W^{1,2} , \,\, \int \rho_0 dx = 1,  \,\, \sigma_0 \in L^1 \cap W^{1,2} \,\, \mathrm{positive}\,\, \mathrm{definite}, \\
\int \rho_0 \left ( |\Lambda(x) |^2 +  \log \rho_0 + \log | \sigma_0 | + \frac { |\sigma_0 | ^2 } {\rho_0 ^2 } \right ) dx  < \infty \\
\left (\frac{1}{\rho_0} \right )^{p-1} \sum_{s \in S_p} \prod_{i=1} ^p \sigma_{i , s(i) }  \in W^{1,2} , \,\, 0 \le p \le 3, \,\, \in L^2 , \,\, 0 \le p \le 8,
\end{gathered} \label{initialcond}
\end{equation}
the distribution $f_0$ given by
\begin{equation}
f_0 (x,m) = \frac{ \rho_0 (x)^2 }{2 \pi \sqrt{\det \sigma_0 (x)}} \exp \left (-\frac{1}{2} m^T \rho_0 (x) {\sigma_0 }^{-1} (x)  m \right ) \label{gaussian}
\end{equation}
satisfies
\begin{equation}
\begin{gathered}
\int_{\mathbb{R}^2} f_0 (x,m) dm = \rho_0 (x), \\
\int_{\mathbb{R}^2} m \otimes m f_0 (x,m) dm = \sigma_0 (x)
\end{gathered} \label{distributionconstraints}
\end{equation}
and $(u_0, f_0)$ satisfies the condition (\ref{initcondmicromacro}). By the uniqueness of the solution of diffusive Oldroyd-B system, we conclude that the for the solution $(u, f)$ of (\ref{Kinetic}) with initial data $(u_0, f_0)$, $(u, \sigma[f] - \rho[f] \mathbb{I}, \rho[f])$ is the solution of (\ref{Oldroyd}) with initial data $(u_0, \sigma[f_0] - \rho[f_0] \mathbb{I}, \rho[f_0] ) = (u_0, \tau_0, \rho_0 )$.

\section{Decay of stress and Boundedness of velocity}

In this section, we prove the following theorem.
\begin{theorem}
Suppose that $u_0 \in  \mathbb{P} W^{2,2}$ and $\rho_0, \sigma_0$ satisfies (\ref{initialcond}) and $\rho_0 \in L^\infty$. Then the system (\ref{Oldroyd}) has global unique solution whose initial data is $u(0) = u_0, \rho (0) = \rho_0, \tau (0) = \sigma_0 - \rho_0 \mathbb{I}$. Moreover, 
\begin{equation}
\begin{gathered}
\norm{u}_{L^\infty (0, \infty ; W^{2,2} ) \cap L^2 (0, \infty; W^{3,2} ) } + \norm{\tau}_{L^2 (0, \infty; W^{2,2} ) }  + \norm{\rho}_{L^\infty (0, \infty; W^{1,2} ) \cap L^2 (0, \infty; W^{2,2} ) } \le C_{*}, \\
\norm{\tau(t) }_{W^{1,2} } +  \norm{\tau(t)}_{L^1} \le \frac{C_{*}}{(t+1)}, \,\, \norm{\rho(t)}_{L^2 } ^2 + \norm{\rho(t) }_{L^\infty} \le \frac{C_{*} }{(t+1) }, \\
\end{gathered}
\end{equation}
where $C_*$ depends only on the initial data and coefficients. \label{Boundedness}
\end{theorem}
\begin{remark}
In the periodic case $(x, t) \in \mathbb{T}^2 \times [0, \infty )$, the same argument in this section we obtain the corresponding result
\begin{equation}
\norm{u}_{L^\infty (0, \infty ; W^{2,2} ) \cap L^2 (0, \infty; W^{3,2} ) } + \norm{\tau}_{L^\infty (0, \infty; W^{1,2} )  \cap L^2 (0, \infty; W^{2,2} ) } \le C_{*}.
\end{equation}
\end{remark}
One difference between the system in \cite{MR2989441} and system (\ref{Oldroyd}) is the presence of diffusion in the evolution of $\rho$, and this allows us to use uniform decay result given in \cite{MR856876}. In this section, $C$ denotes constants that do not depend on the size of initial data or parameters $\nu_1, \nu_2, \mu$, and  $k$. The constants depending on them are denoted by $C_1, C_2$ and so on. First we briefly review the technique in \cite{MR571048} and \cite{MR856876} to show the decay of $\rho$. From the fourth equation of (\ref{Oldroyd}) we get
\begin{equation}
\frac{d}{dt} \norm{\rho(t)}_{L^2} ^2 + 2 \nu_2 \norm{\nabla_x \rho (t)}_{L^2} ^2 = 0.
\end{equation}
Then by Plancherel's theorem we have
\begin{equation}
\frac{d}{dt} \left ( \int \left | \hat{\rho} \right |^2 dx \right ) + \frac{2}{t+1} \int \left | \hat{\rho} \right |^2 dx  \le \frac{2}{t+1} \int_{A(t)} \left | \hat{\rho} \right |^2 d\xi
\end{equation}
where 
\begin{equation}
A(t) = \left \{ \xi : |\xi| \le \left ( \frac{2}{\nu_2 (t+1) } \right ) ^{\frac{1}{2}} \right \}.
\end{equation}
Noting that 
\begin{equation}
| \hat{\rho} (\xi, t) |^2 \le \norm{\rho}_{L^1} ^2 = \norm{\rho_0 }_{L^1} ^2
\end{equation}
and multiplying $ (t+1)^2$ to each side and integrating in time gives
\begin{equation}
\norm{\rho(t)}_{L^2} ^2 \le \frac{C}{(t+1)} \left ( \norm{\rho_0}_{L^2} ^2 + \frac{1}{\nu_2} \norm{\rho_0}_{L^1} ^2 \right ).
\end{equation}
For higher $L^p$ norm, let $p = 2^s$, $q = 2^{s-1}$, with $s \ge 2$. Multiplying $p \rho^{p-1}$ to the same equation and integrating, and applying integration by parts we get
\begin{equation}
\frac{d}{dt} \int |\rho|^p dx = - \nu_2 p(p-1) \int |\nabla_x \rho |^2 \rho^{p-2} dx
\end{equation} 
and since 
\begin{equation}
|\nabla_x \rho|^2 \rho^{p-2} = |\rho^{q-1} \nabla_x \rho|^2 = \frac{1}{q^2} | \nabla_x (\rho ^q ) |^2
\end{equation}
and $\frac{p(p-1)}{q^2} \ge 1$ for $s \ge 2$ we have
\begin{equation}
\frac{d}{dt} \int |\rho^q|^2 dx + \nu_2 \int |\nabla_x \left ( \rho^q \right ) | ^2 dx \le 0.
\end{equation}
Applying the decay estimate for $L^2$ repeatedly, and noting that by maximum principle $\norm{\rho}_{L^\infty} \le \norm{\rho_0}_{L^\infty}$ we have
\begin{equation}
\norm{\rho(t)}_{L^\infty} \le \frac{C}{t+1} \max_{1 \le r \le \infty} \left (1 + \norm{\rho_0}_{L^r} \right ). \label{density}
\end{equation}
We establish estimates on $u$ and $\tau$. First we start with a simple lmma. 
\begin{lemma}
Suppose that $h : [0, \infty ) \rightarrow [0, \infty )$ is a non-increasing function such that $h(0) <\infty$ and $\lim_{t \rightarrow \infty} \frac{h'(t) } {h(t) } = 0$. Then
\begin{equation}
\int_0 ^t e^{-2k(t-s)} h(s) ds \le C h(t).
\end{equation} \label{exponential}
\end{lemma}
\begin{proof}
We only need to establish the bound for $t \rightarrow \infty$, which is obvious by L'H\^{o}pital.
\end{proof}
\begin{proof}[Proof of Theorem \ref{Boundedness} ]
First, from the free energy estimate (\ref{freeenergy}) we get
\begin{equation}
\norm{u(t)}_{L^2} ^2 + 2 \nu_1 \int_0 ^t \norm{\nabla_x u (s) }_{L^2} ^2 ds \le C_1 \label{estimate1}
\end{equation}
where $C_1$ depends only on initial data. Next, by multiplying $\tau$ and integrating the third equation, we have
\begin{equation}
\frac{d}{dt} \norm{\tau}_{L^2} ^2 + \nu_2 \norm{\nabla_x \tau}_{L^2} ^2 + 4k \norm{\tau}_{L^2} ^2 \le \norm{\nabla_x u}_{L^2} \norm{\tau}_{L^4} ^2 + \norm{\nabla_x u}_{L^2} \norm{\tau}_{L^2}   \norm{\rho}_{L^\infty }.
\end{equation}
Then using (\ref{density}) and Ladyzhenskaya's inequality 
\begin{equation}
\norm{\tau}_{L^4} ^2 \le C \norm{\tau}_{L^2} \norm{\nabla_x \tau}_{L^2}
\end{equation}
valid in two dimension, we obtain
\begin{equation}
\frac{d}{dt} \norm{\tau }_{L^2} ^2 + \nu_2 \norm{\nabla_x \tau}_{L^2} ^2 + (4k - \frac{C}{\nu_2} \norm{\nabla_x u}_{L^2 } ^2 ) \norm{\tau}_{L^2} ^2 \le \frac{C_2}{(1+t)^2} \label{estimate2_1}
\end{equation} 
where $C_2 = C \nu_2 \max_{1 \le r \le \infty} (1 + \norm{\rho_0}_{L^r} )^2$. Putting
\begin{equation}
A(t) = 4k - \frac{C}{\nu_2} \norm{\nabla_x u (t) }_{L^2 } ^2
\end{equation} 
then we get 
\begin{equation}
\norm{\tau(t)}_{L^2} ^2 \le \exp \left (-\int_0 ^t A(r) dr \right ) \norm{\tau_0}_{L^2} ^2 + \int_0 ^t \exp \left ( - \int_s ^t A(r) dr \right ) \frac{C_2}{(1+s)^2 } ds
\end{equation}
but from (\ref{estimate1}) and Lemma \ref{exponential} we have
\begin{equation}
\norm{\tau(t)}_{L^2} ^2 \le C C_2 \exp \left ( \frac{C C_1}{\nu_2 } \right ) \frac{1}{(1+t)^2} + \exp \left ( \frac{C C_1 }{\nu_2} - 4kt \right ) \norm{\tau_0 }_{L^2} ^2 \le \frac{C_3}{(1+t)^2}. \label{estimate2_2}
\end{equation}
Also, estimates (\ref{estimate2_1}) and (\ref{estimate2_2}) give us
\begin{equation}
\begin{gathered}
\nu_2 \int_0 ^t \norm{\nabla_x \tau(s) }_{L^2} ^2 ds + 4k \int_0 ^t \norm{\tau(s) }_{L^2} ^2 ds \le C_4 = \norm{\tau_0}_{L^2} ^2 + \int_0 ^t  \frac{C C_3 \norm{\nabla_x u}_{L^2} ^2 }{ \nu_2 (1+s)^2} ds + \int_0 ^t \frac{C_2 ds} {(1+s) ^2}.
\end{gathered} \label{estimate2_3}
\end{equation}
By taking curl to the first equation of the system (\ref{Oldroyd}), we obtain
\begin{equation}
\partial_t \omega + u \cdot \nabla_x \omega = \nu_1 \Delta_x \omega + \mu \nabla_x ^{\perp} \cdot \nabla_x \cdot \tau \label{vorticityeq}
\end{equation}
where $\omega = \partial_1 u_2 - \partial_2 u_1$ is the vorticity. Then we have
\begin{equation}
\frac{d}{dt} \norm{\omega}_{L^2} ^2 + \nu_1 \norm{\nabla_x \omega}_{L^2} ^2 \le \frac{\mu^2}{\nu_1} \norm{\nabla_x \tau}_{L^2} ^2,
\end{equation}
and by (\ref{estimate2_3}) we have
\begin{equation}
\norm{\omega(t)}_{L^2} ^2 + \nu_1 \int_0 ^t \norm{\nabla_x \omega(s)}_{L^2} ^2 ds \le \frac{\mu^2 C_4}{\nu_1 \nu_2} = C_5. \label{estimate3}
\end{equation}
To estimate $\norm{ \nabla_x \tau}_{L^2} ^2$, we multiply $- \Delta_x \tau$ to the third equation of (\ref{Oldroyd}) and integrate: we have
\begin{equation}
\begin{gathered}
\frac{1}{2} \frac{d}{dt} \norm{\nabla_x \tau}_{L^2} ^2 + 2k \norm{\nabla_x \tau}_{L^2} ^2 + \nu_2 \norm{ \Delta_x \tau}_{L^2} ^2 = - \int (\nabla_x u) \cdot (\nabla_x \tau) \cdot (\nabla_x \tau) dx \\
+ \int \left ( (\nabla_x u) \tau + \tau (\nabla_x u)^T \right ) (- \Delta_x \tau) dx 
+ \int \rho \left ( (\nabla_x u) + (\nabla_x u )^T \right ) (- \Delta_x \tau) dx.
\end{gathered}
\end{equation}
The first term in the right hand side is bounded by, using Ladyzhenskaya's inequality and Young's inequality,
\begin{equation}
\left | \int (\nabla_x u) \cdot (\nabla_x \tau) \cdot (\nabla_x \tau) dx \right | \le \frac{C}{\nu_2} \norm{\nabla_x u}_{L^2} ^2 \norm{\nabla_x \tau}_{L^2} ^2 + \frac{\nu_2}{8} \norm{\Delta_x \tau}_{L^2} ^2.
\end{equation}
We apply the integration by parts, and then Ladyzhenskaya's inequality to the second term to obtain
\begin{equation}
\begin{gathered}
|\int \left ( (\nabla_x u) \tau + \tau (\nabla_x u)^T \right ) (- \Delta_x \tau ) dx | \\
 \le \frac{C}{\nu_2} \norm{\nabla_x u}_{L^2} ^2 \norm{\nabla_x \tau}_{L^2} ^2 + \frac{\nu_2}{8} \norm{\Delta_x \tau}_{L^2} ^2 + C \norm{\nabla_x \omega}_{L^2} \norm{\tau}_{L^2} ^{\frac{1}{2}} \norm{\nabla_x \tau}_{L^2} \norm{\Delta_x \tau}_{L^2} ^{\frac{1}{2}}
\end{gathered}
\end{equation}
and by applying Young's inequality we can bound the second term by
\begin{equation}
\frac{C}{\nu_2} \norm{\nabla_x u}_{L^2} ^2 \norm{\nabla_x \tau}_{L^2} ^2 + \frac{\nu_2}{4} \norm{\Delta_x \tau}_{L^2} ^2 + C \norm{\nabla_x \omega}_{L^2} ^2 \norm{\nabla_x \tau}_{L^2} ^2 + \frac{1}{\nu_2} \norm{\tau}_{L^2} ^2. \label{improve1}
\end{equation}
The last term is bounded by
\begin{equation}
\frac{C}{\nu_2} \norm{\rho}_{L^\infty} ^2 \norm{\nabla_x u}_{L^2} ^2 + \frac{\nu_2} {8} \norm{\Delta_x \tau}_{L^2} ^2  \le \frac{C_2 C_5} {\nu_2 ^2 (1+t)^2 } + \frac{\nu_2}{8} \norm{\Delta_x \tau}_{L^2} ^2. \label{improve2}
\end{equation}
To sum up, we have
\begin{equation}
\begin{gathered}
\frac{d}{dt} \norm{\nabla_x \tau}_{L^2} ^2 + \left ( 4k - \frac{C}{\nu_2} \norm{\nabla_x u}_{L^2} ^2 - C \norm{\nabla_x \omega}_{L^2 } ^2 \right ) \norm{\nabla_x \tau}_{L^2} ^2 + \nu_2 \norm{ \Delta_x \tau}_{L^2} ^2 \\
\le \frac{C}{\nu_2} \norm{\tau}_{L^2} ^2 + \frac{C}{\nu_2} \norm{\rho}_{L^\infty}^2 \norm{\nabla_x u}_{L^2} ^2 \le \frac{C_2 C_5}{\nu_2 ^2 (1+t)^2 } + \frac{1}{\nu_2} \norm{\tau}_{L^2} ^2. 
\end{gathered} \label{estimate4}
\end{equation}
Applying Gr\"{o}nwall with Lemma \ref{exponential} we have
\begin{equation}
\begin{gathered}
\norm{\nabla_x \tau (t)}_{L^2} ^2 \le e^{ \frac{CC_1}{\nu_1 \nu_2} + \frac{C C_3}{\nu_1} } \left (C \left ( \frac{C_2 C_5}{\nu_2 ^2} + \frac{C_3}{\nu_2} \right ) \frac{1}{(1+t)^2 } +  e^{ -4kt} \norm{\nabla_x \tau_0 } _{L^2} ^2 \right ) \le \frac{C_6}{(1+t)^2}
\end{gathered}
\end{equation}
and
\begin{equation}
\begin{gathered}
4k \int_0 ^t \norm{\nabla_x \tau (s) }_{L^2} ^2ds + \nu_2 \int_0 ^t \norm{ \Delta_x \tau (s) }_{L^2} ^2 ds \\
\le \norm{\nabla_x \tau_0}_{L^2} ^2 + \left ( \frac{C (C_1+C_5) C_6}{\nu_1} + \frac{C C_3}{\nu_2} + \frac{C C_2 C_5}{\nu_2 ^2} \right ) = C_7.
\end{gathered}
\end{equation}
We can also get an estimate for $\norm{\nabla_x \omega}_{L^2} ^2$; by multiplying $-\Delta_x \omega$ to the vorticity equation (\ref{vorticityeq}) and integrating, we have
\begin{equation}
\frac{1}{2} \frac{d}{dt} \norm{\nabla_x \omega(t)}_{L^2} ^2 + \frac{3}{4} \nu_1  \norm{\Delta_x \omega(t) } _{L^2} ^2  \le \frac{C}{\nu_1} \norm{\Delta_x \tau }_{L^2} ^2 + \norm{\nabla_x \omega}_{L^p} ^2 \norm{\nabla_x u}_{L^q} \label{gradomega}
\end{equation}
where $\frac{2}{p}+\frac{1}{q} = 1$. Our choice of $p$ and $q$ are $p = q = 3$. By the Gagliardo-Nirenberg interpolation inequality
\begin{equation}
\begin{gathered}
\norm{\nabla_x \omega}_{L^3} \le C \norm{\Delta_x \omega}_{L^2} ^{\frac{2}{3}} \norm{\omega}_{L^2} ^{\frac{1}{3} }, \,\, \norm{\nabla_x u}_{L^3} \le C \norm{\Delta_x u}_{L^2} ^{\frac{2}{3}} \norm{u}_{L^2} ^{\frac{1}{3} }, 
\end{gathered}
\end{equation}
and Young's inequality, we have
\begin{equation}
\begin{gathered}
\norm{\nabla_x \omega}_{L^3} ^2 \norm{\nabla_x u }_{L^3 } \le C \norm{\Delta_x \omega}_{L^2} ^{\frac{4}{3} } \norm{\nabla_x \omega} _{L^2} ^{\frac{2}{3} } \norm{\omega}_{L^2} ^{\frac{2}{3} } \norm{u}_{L^2} ^{\frac{1}{3} } \\
\le \frac{\nu_1}{4} \norm{\Delta_x \omega}_{L^2} ^2 + \frac{C}{\nu_1 ^2} \norm{\omega}_{L^2} ^2 \norm{u}_{L^2} \norm{\nabla_x \omega}_{L^2} ^2
\end{gathered} \label{GNyoung}
\end{equation}
to obtain
\begin{equation}
\begin{gathered}
\frac{d}{dt} \norm{\nabla_x \omega  }_{L^2 } ^2 + \nu_1 \norm{\Delta_x \omega }_{L^2} ^2 \le \frac{C}{\nu_1} \norm{\Delta_x \tau}_{L^2} ^2 + \frac{C}{\nu_1 ^2} \norm{\omega}_{L^2} ^2 \norm{u}_{L^2} \norm{\nabla_x \omega}_{L^2} ^2.
\end{gathered} \label{estimate5}
\end{equation}
Since $\norm{u(t)}_{L^2} \le \sqrt{C_1}, \int_0 ^T \norm{\omega(t) }_{L^2} ^2 dt \le \frac{C C_1}{\nu_1}$, we have
\begin{equation}
\sup_{t \in [0, T] } \norm{\nabla_x \omega (t) }_{L^2 } ^2 + \nu_1 \int_0 ^T \norm{\Delta_x \omega (t) } _{L^2 } ^2 dt \le \exp \left ( \frac{C C_1^{\frac{3}{2} } } {\nu_1 ^3 }  \right ) \left ( \norm{\nabla_x \omega (0 ) }_{L^2} ^2 + \frac{C C_7} {\nu_1 \nu_2} \right ) = C_8.
\end{equation}
Moreover, we have $L^1$ decay of $\tau$. From
\begin{equation}
\frac{d}{dt} \int \mathrm{Tr} \tau = 2 \int \mathrm{Tr} \left ( (\nabla_x u ) \tau \right ) - 2k \int \mathrm{Tr} \tau
\end{equation}
we obtain
\begin{equation}
\frac{d}{dt} \left ( e^{2kt} \int \mathrm{Tr} \tau \right ) \le e^{2kt} \norm{\nabla_x u(t)}_{L^2} \norm{\tau (t) }_{L^2} \le e^{2kt} \frac{C \sqrt{C_5 C_3}}{(1+t)}  = \frac{C_9 e^{2kt} }{1+t}. \label{estimate6}
\end{equation}
Therefore,
\begin{equation}
\int \mathrm{Tr} \tau (t) dt \le e^{-2kt} \int \mathrm{Tr} \tau_0 + \frac{C C_9}{(1+t)} \le \frac{C C_9}{1+t}.
\end{equation}
Finally, by multiplying $-\Delta_x \rho$ to the fourth equation of (\ref{Oldroyd}) and integrating, and using Ladyzhenskaya's inequality we have
\begin{equation}
\frac{d}{dt} \norm{\nabla_x \rho (t) }_{L^2} ^2 + \nu_2 \norm{\Delta_x \rho (t) }_{L^2} ^2 \le \frac{C}{\nu_2} \norm{\nabla_x u}_{L^2} ^2 \norm{\nabla_x \rho}_{L^2} ^2,
\end{equation}
and by Gr\"{o}nwall 
\begin{equation}
\sup_{t \in [0, T] } \norm{\nabla_x \rho(t) }_{L^2} ^2 + \nu_2 \int_0 ^T \norm{\Delta_x \rho (t) }_{L^2} ^2 dt \le \exp \left ( \frac{C C_1} {\nu_1 \nu_2 } \right ) \norm{\nabla_x \rho (0) } _{L^2} ^2 = C_{10} .
\end{equation}
\end{proof}
\begin{remark} \label{fasterdecay}
We will make an improvement for (\ref{estimate2_1}), (\ref{estimate4}) and (\ref{estimate6}). For (\ref{estimate2_1}) we used the bound
\begin{equation}
\int \rho ((\nabla_x u) +(\nabla_x u)^T ) \tau dx \le C \nu_2 \norm{\rho}_{L^\infty}^2 + \frac{C}{\nu_2} \norm{\nabla_x u}_{L^2} ^2 \norm{\tau}_{L^2} ^2.
\end{equation}
We may instead use
\begin{equation}
\int \rho ((\nabla_x u) +(\nabla_x u)^T ) \tau dx \le \frac{C}{k} \norm{\rho}_{L^\infty} ^2 \norm{\nabla_x u}_{L^2} ^2 + k \norm{\tau}_{L^2} ^2.
\end{equation}
Then we have
\begin{equation}
\frac{d}{dt} \norm{\tau}_{L^2} ^2 + \nu_2 \norm{\nabla_x \tau}_{L^2} ^2 + \left (3k - \frac{C}{\nu_2} \norm{\nabla_x u}_{L^2} ^2 \right ) \norm{\tau}_{L^2 }^2 \le \frac{C_2}{k} \norm{\nabla_x u}_{L^2} ^2 \frac{1}{(1+t)^2 } \label{estimate2_f}.
\end{equation}
In the later section, we show that $\norm{\nabla_x u (t)}_{L^2} ^2 \le \frac{C}{(t+1)}$ in time, so (\ref{estimate2_f}) gives faster decay of $\norm{\tau}_{L^2} ^2$, $\norm{\nabla_x \tau}_{L^2 } ^2 $ by (\ref{estimate4}), and $\norm{\tau}_{L^1}$ by (\ref{estimate6}). 
\end{remark}

\section{Decay of the strong solution}

In this section, we prove the following theorem.
\begin{theorem} \label{decaystrong}
Suppose that $u_0 \in  \mathbb{P} W^{2,2}$ and $\rho_0, \sigma_0$ satisfies (\ref{initialcond}) and $\rho_0 \in L^\infty$. Furthermore, suppose that $u_0 \in L^1$. Then there is a constant $C_{**}$ depending on initial data and parameter such that
\begin{equation}
\norm{u(t)}_{W^{2,2} } \le \frac{C_{**} }{(t+1) ^{\frac{1}{2}} },\,\, \norm{\tau (t) }_{W^{1,2} \cap L^1 } \le \frac{C_{**} }{(t+1)^{\frac{3}{2} } } .
\end{equation}
\end{theorem}
\begin{proof}
We first follow the methods presented in \cite{MR1356749}. Let
\begin{equation}
S(t) = \left \{ \xi : | \xi | \le \left ( \frac{2}{\nu_1 (t+1) } \right )^{\frac{1}{2} } \right \}.
\end{equation}
We apply the estimate for the decay of $\norm{\omega (t) }_{L^2} ^2$. From the vorticity equation we have
\begin{equation}
\begin{gathered}
\frac{d}{dt} \norm{\omega (t) }_{L^2} ^2 \le - \nu_1 \norm{ \nabla_x \omega (t) }_{L^2} ^2 + \frac{C \mu^2}{\nu_1} \norm{\nabla_x \tau (t) }_{L^2} ^2, 
\end{gathered}
\end{equation}
and by Plancherel's theorem
\begin{equation}
\begin{gathered}
\frac{d}{dt} \norm{\omega(t) }_{L^2} ^2 \le - \nu_1 \int_{\mathbb{R}^2} |\xi|^2 | \hat{\omega}(t) |^2 d \xi + \frac{C C_*}{\nu_1 (t+1) ^2 } \le - \nu_1 \int_{S(t) ^c} |\xi|^2 | \hat{\omega}(t) |^2 d \xi + \frac{C C_*}{\nu_1 (t+1) ^2 } \\
\le - \frac{2}{t+1} \int_{S(t) ^c} |\hat{u} (t) |^2 d\xi +\frac{C C_*}{\nu_1 (t+1) ^2 } \le - \frac{2}{t+1} \norm{\omega(t)}_{L^2} ^2 + \frac{2}{t+1} \int_{S(t) } |\hat{\omega} (t) |^2 d\xi + \frac{C C_*}{\nu_1 (t+1) ^2 }.
\end{gathered} 
\end{equation}
To summarize, we have
\begin{equation}
\begin{gathered}
\frac{d}{dt} \norm{\omega (t) }_{L^2} ^2 \le -\frac{2}{t+1} \norm{\omega(t)}_{L^2} ^2 + \frac{2}{t+1} \int_{S(t) } \left | \hat{\omega} (t) \right | ^2 d \xi + \frac{C C_{**} }{(t+1)^2}. 
\end{gathered} \label{wdecay1}
\end{equation} 
We need a pointwise estimate of $\hat{\omega} (\xi, t)$. For that purpose we investigate $\hat{u} (\xi, t) $. The phase space version of the velocity equation of (\ref{Oldroyd}) reads 
\begin{equation}
\partial_t \widehat{u} + \nu_1 |\xi|^2 \widehat{u} = - \xi \cdot \left ( \mathbb{I} - \frac{\xi \otimes \xi}{|\xi|^2} \right ) \mathcal{F} (u \otimes u - \mu \tau ),
\end{equation}
and from $\norm{\mathcal{F} (u \otimes u ) (t) }_{L^\infty} \le \norm{u(t) }_{L^2} ^2 \le C_*$ and $\norm{\mathcal{F} (\tau) (t) }_{L^\infty} \le \norm{\tau (t) }_{L^1} \le \frac{C_*}{(t+1)}$, we have
\begin{equation}
\begin{gathered}
\hat{u} (\xi, t) = e^{-\nu_1 |\xi|^2 t} \hat{u} (\xi, 0)  - \int_0 ^t e^{-\nu_1 |\xi|^2 (t-s) } \xi \cdot \left ( \mathbb{I} - \frac{\xi \otimes \xi}{|\xi|^2} \right ) \mathcal{F} (u \otimes u - \mu \tau ) (s) ds, \\
|\widehat{u} (\xi, t) | \le C_* ( \frac{1}{|\xi|} + 1)
\end{gathered}
\end{equation}
given that $u_0 \in L^1$. Therefore, since $|\hat{\omega} (\xi, t) | \le |\xi| |  \hat{u} (\xi, t) | \le C_* \left (|\xi| + 1 \right ) $. Therefore, $\int_{S(t) } |\hat{\omega} (t) |^2 d\xi \le \frac{C_*}{(t+1)^2 \nu_1 }$, and by multiplying $(t+1)^2$ to (\ref{wdecay1}) we obtain
\begin{equation}
\frac{d}{dt} \left ( (t+1) ^2 \norm{\omega(t) }_{L^2} ^2 \right ) \le C C_{**},
\end{equation}
and therefore
\begin{equation}
\norm{\omega(t)}_{L^2} ^2 \le \frac{C_{**}}{ (t+1)}. \label{wdecay2}
\end{equation}
In addition, as we mentioned in Remark \ref{fasterdecay}, this gives a better decay rate for $\tau$: 
\begin{equation}
\begin{gathered}
\norm{\tau(t)}_{L^2}  \le \frac{C_{**}}{(t+1)^{\frac{3}{2} } } , \norm{\nabla_x \tau(t)}_{L^2}  \le \frac{C_{**}}{(t+1)^{\frac{3}{2} } } , \norm{\tau(t)}_{L^1} \le \frac{C_{**}}{(t+1)^{\frac{3}{2} } }.
\end{gathered}
\end{equation}
Finally, from (\ref{estimate5}), using $\norm{\nabla_x \omega}_{L^2} ^2 \le C_*$ and previous estimates we have
\begin{equation}
\frac{C}{\nu_1 ^2} \norm{\omega (t) }_{L^2} ^2 \norm{u (t) }_{L^2} \norm{\nabla_x \omega (t) }_{L^2} ^2 \le \frac{C C_{**} }{(t+1) ^\frac{3}{2} } \label{estimate7}
\end{equation}
and adding multiples of (\ref{estimate4}) to (\ref{estimate5}) to cancel $\frac{C}{\nu_1} \norm{\Delta_x \tau}_{L^2} ^2$ in the right side of (\ref{estimate5}), we obtain
\begin{equation}
\begin{gathered}
\frac{d}{dt} \left ( \norm{\nabla_x \omega (t) }_{L^2} ^2 + \frac{C}{\nu_1 \nu_2 } \norm{\nabla_x \tau (t) }_{L^2} ^2 \right ) + \frac{C}{\nu_1 \nu_2 } \left ( 4k - \frac{C}{\nu_2} \norm{\nabla_x u (t) }_{L^2} ^2 - C \norm{\nabla_x \omega (t) }_{L^2} ^2 \right ) \norm{\nabla_x \tau (t) }_{L^2 } ^2 \\
+ \frac{C}{\nu_1 } \norm{\Delta_x \tau (t) }_{L^2} ^2 + \nu_1 \norm{\Delta_x \omega (t) }_{L^2} ^2 \le \frac{C_{**}}{(1+t)^3 } + \frac{C_{**} C_*}{(1+t)^{\frac{3}{2} } }.
\end{gathered}
\end{equation}
Then applying the same Fourier splitting technique, together with the pointwise bound
\begin{equation}
\left | \mathcal{F} (\nabla_x \tau ) (\xi, t) \right | \le |\xi| \norm{\tau (t) }_{L^1} \le \frac{C_* |\xi| } {(1+t) }, \,\, \left | \mathcal{F} (\nabla_x \omega ) (\xi, t ) \right | \le C_* |\xi| (1 + |\xi | )
\end{equation}
we obtain
\begin{equation}
\frac{d}{dt} \left ( (t+1) ^2 \left ( \norm{\nabla_x \omega (t) }_{L^2} ^2 + \frac{C}{\nu_1 \nu_2 } \norm{\nabla_x \tau (t) }_{L^2} ^2 \right ) \right ) \le C_{**} \left ( \frac{1}{(t+1) } + (t+1) ^{\frac{1}{2} } \right ),
\end{equation}
and from this we conclude that
\begin{equation}
\norm{\nabla_x \omega (t) }_{L^2} ^2 \le \frac{C_{**} }{(t+1) ^{\frac{1}{2} } }. \label{estimate8}
\end{equation}
We put (\ref{estimate8}) into (\ref{estimate7}), and then we repeat the last estimate to obtain
\begin{equation}
\norm{\nabla_x \omega (t) }_{L^2} ^2 \le \frac{C_{**} } {(t+1) }.
\end{equation}
Now we are ready to prove decay of $\norm{u(t)}_{L^2}$. We follow the argument in \cite{MR881519}. 
\begin{definition}
For given $u_0 \in L^2 (\mathbb{R}^2)$ and $f \in L^1 (0, +\infty; L^2 (\mathbb{R}^2) )$, we say that $(u_0, f )$ belong to the set $D_\alpha ^{(n)}$ with $\alpha \ge 0$ if there is a constant $C$ such that
\begin{equation}
\norm{v(t)}_{L^2} ^2 + (1+t)^2 \norm{f(t)}_{L^2} ^2 \le C (1+t)^{-\alpha}
\end{equation}
where $v(t)$ is the solution of the heat equation
\begin{equation}
\begin{gathered}
\partial_t v - \nu_1 \Delta_x v = f,  \,\, v(0) = u_0.
\end{gathered} \label{heat}
\end{equation}
\end{definition}
\begin{theorem} [\cite{MR881519}] \label{Wiegner}
Let $u$ be a strong solution of Navier-Stokes system
\begin{equation}
\begin{gathered}
\partial_t u + u \cdot \nabla_x u = - \nabla_x p + \nu_1 \Delta_x u + f, \\
\nabla_x u = 0, \,\,  u(0) = u_0 , \\
u \in C_w ([0, +\infty ), L^2 (\mathbb{R}^2 )) \cap L^2 _{loc} ( 0, +\infty ; W^{1,2} (\mathbb{R}^2 ) ).
\end{gathered}
\end{equation}
Let $u_0 \in L^2 (\mathbb{R}^2 ), \,\, f \in L^1 (0, +\infty; L^2 (\mathbb{R}^2 ) )$. Then we have the following.
\begin{enumerate}
\item $\norm{u(t)}_{L^2} \rightarrow 0$ for $t \rightarrow \infty$.
\item If $(u_0, f) \in D_{1} ^{(2)}$, then $\norm{u(t)}_{L^2} ^2 \le C (1+t)^{-1}$.
\item Furthermore, the solution $u$ is asymptotically equivalent to the heat system (\ref{heat}) in the sense that
\begin{equation}
\norm{u(t) - v(t) }_{L^2} ^2 \le C \left (\log (t+e) \right )^2 (t+1)^{-2}.
\end{equation}
\end{enumerate}
\end{theorem}
The proof of Theorem \ref{Wiegner} is based on the Fourier splitting technique. The proof of Theorem \ref{decaystrong} is completed by the following Lemma.
\end{proof}
\begin{lemma} \label{datacond}
Suppose that $u_0, \tau_0, \rho_0$ satisfy the initial conditions in Theorem \ref{decaystrong}. Then $(u_0,\nabla_x \cdot \tau)$ belong to the set $D_1 ^{(2)}$, where $(u, \tau, \rho)$ is the solution of (\ref{Oldroyd}) with initial data $(u_0, \tau_0, \rho_0).$
\end{lemma}
\begin{proof}[proof of the Lemma]
Since $\norm{\nabla_x \tau (t) }_{L^2} \le \frac{C_{**}}{(t+1) ^{\frac{3}{2} } }, $ $\nabla_x \tau \in L^1 (0, +\infty; L^2 (\mathbb{R}^2 ) ^2 )$. Furthermore, by the same Fourier splitting technique,
\begin{equation}
\frac{d}{dt} \norm{v(t)}_{L^2} ^2 \le -\frac{2}{(t+1)} \norm{v(t)}_{L^2} ^2 + \frac{2}{(t+1)} \int_{S(t)} |\hat{v}(t) |^2 d\xi + \frac{C_{**} }{(t+1) ^3}.
\end{equation}
However, in the Fourier space
\begin{equation}
\begin{gathered}
\partial_t \hat{v} (t) = - \nu_1 |\xi|^2 \hat{v} (t) - i \xi \cdot \hat{\tau}, \\
\hat{v}(t) = e^{-\nu_1 |\xi|^2 t} \hat{u_0} - i \int_0 ^t \xi \cdot e^{-\nu_1 |\xi|^2 (t-s) } \hat{\tau} (s) ds,
\end{gathered}
\end{equation}
and since $|\hat{\tau}(s, \xi ) | \le \norm{\tau(s)}_{L^1} \le \frac{C_{**}}{(s+1) ^{\frac{3}{2}} }$, 
\begin{equation}
|\hat{v}(t, \xi) | \le \norm{u_0}_{L^1} +  C_{**} |\xi|, 
\end{equation}
then we have $\norm{v(t)}_{L^2}^2 \le \frac{C_{**} }{(t+1)}$, as desired.
\end{proof}

\section{Conclusion}

We proved that strong solutions of 2D diffusive Oldroyd-B systems in $\mathbb{R}^2$ decay over time at an algebraic rate, given that the initial data satisfy mild restrictions originating from the kinetic considerations.
\paragraph{Acknowledgements} The author is supported by Samsung scholarship. The author is very grateful to Prof. Peter Constantin for his kind support. The author thanks to Dr. Theo Drivas and Dr. Dario Vincenzi for helpful discussions.

\bibliographystyle{abbrv}
\bibliography{jl2}
\end{document}